\documentclass[10pt,reqno]{amsart}
\usepackage{amsfonts}
\usepackage{amsmath,amsthm,amssymb,amsfonts,amscd}
\usepackage{mathrsfs}
\usepackage{lipsum}
\usepackage{bbding}
\usepackage{graphicx,latexsym}
\usepackage{hyperref}
\usepackage{color}

\newcommand{\bea}{\begin{eqnarray}}
\newcommand{\eea}{\end{eqnarray}}
\newcommand{\bna}{\begin{eqnarray*}}
\newcommand{\ena}{\end{eqnarray*}}

\numberwithin{equation}{section} % ??¡ì¡§¡ã?1?¡§o??????¡§2?¨¤¡§¡èo?

\renewcommand{\thefootnote}{\fnsymbol{footnote}}
\theoremstyle{plain}
\newtheorem{theorem}{Theorem}[section]
\newtheorem{lemma}{Lemma}[section]
\newtheorem{corollary}{Corollary}
\newtheorem{proposition}{Proposition}[section]

\theoremstyle{definition}

\newtheorem{remark}{Remark}

\newcommand\blfootnote[1]{%
  \begingroup
  \renewcommand\thefootnote{}\footnote{#1}%
  \addtocounter{footnote}{-1}%
  \endgroup
}

\begin{document}

\title{On $\rm GL_3$ Fourier coefficients over values of mixed powers}

\author{Qingfeng Sun and Yanxue Yu }

\begin{abstract}
Let $A_{\pi}(n,1)$ be the $(n,1)$-th Fourier coefficient of the Hecke-Maass cusp form $\pi$ for $\rm SL_3(\mathbb{Z})$ and $ \omega(x)$ be a smooth compactly supported function.
In this paper, we prove a nontrivial upper bound for the sum
$$
\sum_{n_1,\cdots,n_\ell,n_{\ell+1}\in \mathbb{Z}_+
\atop n=n_1^r+\cdots+n_{\ell}^r+n_{\ell+1}^s}
A_{\pi}(n,1)\omega\left(n/X\right),
$$
where $r\geq2$, $s\geq 2$ and $\ell\geq 2^{r-1}$ are integers.
\end{abstract}

\thanks{This work was partially supported by the National Natural Science Foundation
 of China (Grant No. 12471005) and the Natural Science Foundation of Shandong Province (Grant No. ZR2023MA003)}

\maketitle


\section{Introduction}

Modular forms were initially discovered and studied for purposes of complex analysis and algebraic geometry,
but they have since played significant roles in the development of many branches of number theory,
such as class field theory, Galois representations, arithmetic of elliptic curves, and so on.
They are now both fundamental tools and interesting research objects in number theory.
The key information of modular forms is encoded in their Fourier coefficients,
which are mysterious arithmetic objects, and it is therefore natural to study the distribution of Fourier coefficients.
Let $F(z)$ be a modular form for $\mathrm{GL}_m$
and let $A_F(n,1,\cdots,1)$ denote its $(n,1,\cdots,1)$-th normalized Fourier coefficient.
The generalized Ramanujan-Petersson conjecture asserts that
$$
|A_F(n,1,\cdots,1)|\leq \tau_m(n),
$$
where $\tau_m(n)$ denotes the divisor function of order $m$,
which is the number of representations of $n$ as the product of $m$ natural numbers.
For holomorphic cusp forms of $\mathrm{GL}_2$, the Ramanujan-Petersson conjecture was established by Deligne \cite{De}, Deligne and Serre \cite{Del-Ser}.
For Hecke-Maass cusp forms of $\mathrm{GL}_m$, where the conjecture remains open,
the best current record estimates are
$$
|A_F(n)| \leq n^{\frac{7}{64}}\tau(n), \quad |A_F(n,1)| \leq n^{\frac{5}{14}}\tau_3(n), \quad
|A_F(n,1,1)| \leq n^{\frac{9}{22}}\tau_4(n),
$$
$$
|A_F (n,1,\ldots,1)| \leq n^{\frac{1}{2}-\frac{1}{m^2+1}}\tau_m(n) \;(m \geq 5),
$$
which are due to Kim and Sarnak \cite{ks} for $2 \leq m \leq 4$
and Luo et al. \cite{LRS1}, \cite{LRS2} for $m \geq 5$.

On the other hand, the Rankin-Selberg theory gives
$$
\sum_{n\leq X}|A_F (n,1,\ldots,1)|^2\ll_F X,
$$
which implies that the Fourier coefficient $A_F(n,1,\cdots,1)$ behaves like a constant on average (see \cite[Remark 12.1.8]{GL1}).
Therefore, in order to explore the distribution of Fourier coefficients at a deeper level,
it is highly valuable and significantly more challenging
to explore the distribution of Fourier coefficients over sparse sequences,
such as the values of an $\ell$-variable nonsingular polynomial $P({\bf x})\in \mathbb{Z}[x_1,\ldots,x_{\ell}]$.

More precisely, we are concerned with the sum
\bna
\mathscr{S}_F(X)=\sum_{{\bf n}\in X\mathcal{B}\bigcap \mathbb{N}^{\ell}}A_F(P({\bf n}),1,\cdots,1)
\ena
for $X\rightarrow \infty$,
where $\mathcal{B}\subset \mathbb{R}^{\ell}$ is an $\ell$-dimensional box
such that $\min_{{\bf x}\in  X\mathcal{B}} P({\bf x}) \geq 0$ for all sufficiently large $X$.
The sum $\mathcal{S}_F(X)$ has been studied in several cases
when $F$ is a holomorphic cusp form of $\mathrm{GL}_2$.
For instance, Blomer \cite{blomer} proved that
\bna
\sum_{n\leq X}A_F(P(n))=c_{F,P} X+O_{F,P,\varepsilon}\big(X^{\frac{6}{7}+\varepsilon}\big)
\ena
for some constant $c_{F,P}\in \mathbb{C}$ and any $\varepsilon>0$,
where $F\in S_{\kappa}(N,\chi)$ is a holomorphic cusp form
of weight $\kappa\geq 4$ and character $\chi$ for $ \Gamma_0(N)$,
and $P(x)\in\mathbb{Z}[x]$ is an integral monic quadratic polynomial.
When $P(x) = x^2+d$ for some $d \in\mathbb{Z}$,
Templier \cite{T} also studied a similar sum
when $F$ is a classical weight 2 modular form of odd square-free level and trivial Nebentypus,
$X$ is about $d^{\frac{1}{2}}$ and $d>0$,
which is intimately related to the equidistribution of Heegner points.
Later Templier and Tsimerman \cite{TT} generalized the results of Blomer \cite{blomer} and Templier \cite{T} to any $\mathrm{GL}_2$ automorphic cuspidal representation.
In another direction, for $P({\bf x})=x_1^2 + x_2^2$,
Acharya \cite{acharya} proved that
\bna
\sum_{P({\bf n})\leq X} A_F(P({\bf n})) \ll_{F,\varepsilon} X^{\frac{1}{2} + \varepsilon}
\ena
for any $\varepsilon>0$,
where $F\in S_{\kappa}(4N,1)$ is a holomorphic cusp form of weight $\kappa$
for $ \Gamma_0(4N)$ with $N \in \mathbb{N}$ and trivial character.
Moreover, Pandey and Vaishya \cite{PandeyVaishya1} recently generalized
Acharya's result by investigating the distribution of $\{A_F(\mathcal{Q}(\underline{x}))\}$, where $\mathcal{Q}$
is a primitive integral positive-definite binary quadratic form of fixed discriminant $D < 0$
with class number $h(D) = 1$, and they recovered Acharya's result with explicit terms involving the conductor aspects,
including the weight and level. As an application, they also proved interesting results on the sign changes of the associated
Fourier coefficients.
Higher moments of classical modular forms associated with binary quadratic forms have also been intensively studied;
see Pandey and Vaishya \cite{PandeyVaishya2}.
For more interesting results, see Zhai \cite{Zhai}, Kumaraswamy \cite{ku}, Xu \cite{Xu},
Liu \cite{Liu}, Hua \cite{H}, Vaishya \cite{V1}, and the references therein.

Compared to the fruitful results of $\mathrm{GL}_2$,
much less has been done for modular forms of $\mathrm{GL}_3$.
Let $\pi$ be a Hecke-Maass cusp form of type $(\nu_1,\nu_2) \in \mathbb{C}^2$ for $\rm SL_3(\mathbb{Z})$
with normalized Fourier coefficients $A_\pi(m,n)$ such that $A_\pi(1,1)=1$.
Recently, Chanana and Singh \cite{CS1} proved that,
for a symmetric positive definite quadratic form $Q(x,y)$,
$$
\mathop{\sum\sum}_{n_1,n_2\leq X}A_\pi(Q(n_1,n_2),1) W_1\left(\frac{n_1}{X}\right) W_2\left(\frac{n_2}{X}\right)
\ll_{\pi,\varepsilon} X^{2-\frac{1}{68}+\varepsilon}
$$
for any $\varepsilon>0$, where $W_1$ and $W_2$ are smooth bump functions supported on the interval $[1,2]$.
Moreover, in \cite{CS2}, they established an asymptotic formula for the sum
\begin{equation*}
\mathop{\sum}_{\substack{1\leq n_1, n_2 \leq X^{1/2} \\ 1 \leq n_3 \leq X^{\frac{1}{s}}}}
A_\pi(Q(n_1,n_2)+n_3^s,1)\mathsf{a}(n_3),
\end{equation*}
where $Q(x,y)\in\mathbb{Z}[x,y]$ is a binary quadratic polynomial
and $\mathsf{a}(n): \mathbb{N} \rightarrow \mathbb{C}$ can be any bounded arithmetic function.

In this paper, we aim to explore more scenarios to gain a deeper understanding
of the distribution of $\mathrm{GL}_3$ Fourier coefficients $A_\pi(m,n)$ over sparse sets.
More precisely, let $\omega(x)$ be a smooth function compactly supported on $[1,2]$ and satisfying
\bna
\omega^{(j)}(x)\ll_j \Delta^j \quad\text{ for}\quad j\geq 0 \quad
\text{and} \quad
\int|\omega^{(j)}(\xi)|\mathrm{d}\xi\ll \Delta^{j-1} \quad \text{for}\quad j\geq1,
\ena
where $1\leq \Delta<X$. We are concerned with the following general sum
\bna
\sum_{n_1,\cdots,n_\ell,n_{\ell+1}\in \mathbb{Z}_+
\atop n=n_1^r+\cdots+n_{\ell}^r+n_{\ell+1}^s}
A_{\pi}(n,1)\omega\left(\frac{n}{X}\right),
\ena
where $r\geq2$, $s\geq 2$ and $\ell\geq2^{r-1}$ are positive integers.
Here $r$ and $s$ may be the same or different.

By Cauchy-Schwarz's inequality and Rankin-Selberg's estimate,
one sees that the trivial bound of the above sum is
$O_{\pi,\varepsilon}\big(X^{\frac{\ell}{r}+\frac{1}{s}+\varepsilon}\big)$.
In this paper, we will prove the following result.

\begin{theorem}\label{Th11}
Let $r\geq2$, $s\geq 2$, $\ell\geq2^{r-1}$ be integers and
$$\theta_0=\min\big\{1/r, 1/s\big\}.$$
For any $\varepsilon>0$, we have
\bna
\sum_{n_1,\cdots,n_\ell,n_{\ell+1}\in \mathbb{Z}_+ \atop n=n_1^r+\cdots+n_{\ell}^r+n_{\ell+1}^s}A_{\pi}(n,1)\omega\left(\frac{n}{X}\right)
\ll_{\pi,\varepsilon}
\begin{cases}
\Delta X^{\frac{\ell}{r}+\frac{1}{s}-1+\varepsilon}+R,
&\text{if \;} \frac{\ell}{r}+\frac{1}{s}\geq \frac{7}{2},\\
\\
\Delta X^{\frac{\ell}{r}+\frac{1}{s}-\left(1-
\left(\frac{7}{2}-\frac{\ell}{r}-\frac{1}{s}\right)\theta_0\right)+\varepsilon}+R,
&\text{if \;} \frac{\ell}{r}+\frac{1}{s}<\frac{7}{2},
\end{cases}
\ena
where
\begin{center}
$R=\begin{cases}
X^{\frac{\ell}{r}+\frac{1}{s}-\left(\frac{\ell}{2^{r-1}}+\frac{1}{2^{s-1}}-1\right)\theta_0+\varepsilon},
&\text{if \,} 2\le r\le 7,\, 2\le s\le 7,\\
\\
X^{\frac{\ell}{r}+\frac{1}{s}-\left(\frac{\ell}{2^{r-1}}+\frac{1}{2s(s-1)}-1\right)\theta_0+\varepsilon},
&\text{if \,} 2\le r\le 7,\, s\ge 8,\\
\\
X^{\frac{\ell}{r}+\frac{1}{s}-\left(\frac{\ell-2^{r-1}}{2r(r-1)}+\frac{1}{2^{s-1}}\right)\theta_0+\varepsilon},
&\text{if \,} r\ge8,\, 2\le s\le 7,\\
\\
X^{\frac{\ell}{r}+\frac{1}{s}-\left(\frac{\ell-2^{r-1}}{2r(r-1)}+\frac{1}{2s(s-1)}\right)\theta_0+\varepsilon},
&\text{if \,} r\ge8,\, s\ge8.\\
\end{cases}$
\end{center}
\end{theorem}

\begin{corollary}
Under the assumptions of Theorem~\ref{Th11}, for $\Delta=1$, we have the following result
\bna
\sum_{n_1,\cdots,n_\ell,n_{\ell+1}\in \mathbb{Z}_+ \atop n=n_1^r+\cdots+n_{\ell}^r+n_{\ell+1}^s}A_{\pi}(n,1)\omega\left(\frac{n}{X}\right)
\ll_{\pi,\varepsilon}
\begin{cases}
X^{\frac{\ell}{r}+\frac{1}{s}-1+\varepsilon}+R,
&\text{if \;} \frac{\ell}{r}+\frac{1}{s}\geq \frac{7}{2},\\
\\
R,
&\text{if \;} \frac{\ell}{r}+\frac{1}{s}<\frac{7}{2},
\end{cases}
\ena
where $R$ is as in Theorem~\ref{Th11}.
\end{corollary}

\begin{remark}
The parameter $\Delta$ controls the oscillation of the weight function $\omega$.
A larger $\Delta$ allows for more oscillation. Normally, we set $\Delta$ equal to 1.
\end{remark}

\begin{remark}
The result of Theorem \ref{Th11} provides a nontrivial estimate for a wide range of parameters,
though it may not be optimal in every case.
For instance, when $r= s= \ell= 2$, the conclusion of Theorem \ref{Th11} becomes trivial.
A more refined analysis is required for such cases to obtain sharper results.
\end{remark}

\begin{remark}
It is worth noting that for the general divisor function $\tau_k(n)$,
which is the Dirichlet coefficient of the simplest degree three $L$-function $\zeta^{k}(z)$,
there has been a long series of works by many number theorists regarding the corresponding sum
$$
\sum_{{\bf n}\in X\mathcal{B}\bigcap \mathbb{N}^{\ell}}\tau_k(P({\bf n}))
$$
for $X\rightarrow \infty$, where $\mathcal{B}\subset \mathbb{R}^{\ell}$
is an $\ell$-dimensional box such that
$\min_{{\bf x}\in X\mathcal{B}} P({\bf x}) \geq 0$ for all sufficiently large $X$;
see Du and Sun \cite{DS}, and the references therein.
\end{remark}

\noindent{\bf Notation.}
Throughout the paper, the letters $d$, $q$, $m$ and $n$,
with or without subscript, denote integers.
The letter $\varepsilon$ is an arbitrarily small positive constant,
not necessarily consistent across occurrences.
The symbol $\ll_{a,b,c}$ denotes that the implied constant depends at most on $a$, $b$ and $c$.

\section{Preliminaries}\label{pre}
\setcounter{equation}{0}

In this section, we will review some basic results associated with $\rm GL_3$ Hecke-Maass cusp forms which will be used in the proof.
Let $\pi$ be a Hecke-Maass cusp form of type $(\nu_1,\nu_2)\in \mathbb{C}^2$ for $\rm SL_3(\mathbb{Z})$,
which has a Fourier-Whittaker expansion with normalized Fourier coefficients $A_\pi(n,m)$
(for more details see \cite{DG}).
From Kim-Sarnak's bounds \cite{ks}, we have
\bea\label{kk}
A_\pi(n,m)\ll |mn|^{\vartheta+\varepsilon}
\eea
for any $\varepsilon>0$, where $\vartheta=5/14$.

We introduce the Langlands parameters $(\alpha_1,\alpha_2,\alpha_3)$,
which are defined as
\bna \label{langland}
\alpha_1= - \nu_1 -2\nu_2 + 1, \; \alpha_2 = -\nu_1 + \nu_2 \; \text{ and } \alpha_3 = 2 \nu_1 + \nu_2 - 1.
\ena
The generalized Ramanujan conjecture asserts that $|\text{Re}(\alpha_j)| = 0,\,1\leq j\leq3$, while the current record bound due to Luo et al. \cite{LRS2} is
\bna\label{re}
|\text{Re}(\alpha_j)|\leq \frac{1}{2}-\frac{1}{10},\quad 1\leq j\leq 3.
\ena

We first recall the Voronoi summation formula for $\rm GL_3$ (see \cite{GL1}, \cite{MS}).
\begin{lemma}\label{gl3voronoi}
Let $A_\pi(n,1)$ be the $(n,1)$-th Fourier coefficient of a Maass cusp form for $\rm SL_3(\mathbb{Z})$.
Suppose that $\phi(x)\in C_c^{\infty}(0, \infty)$.
Let $a,q\in\mathbb{Z}$ with $q\geq1$, $ (a,q)=1$ and $a\overline{a}\equiv1\;(\text{{\rm mod }} q)$.
Then
\bea\label{voronoi}
\sum_{n=1}^{\infty} A_\pi(n,1) e\left(\frac{an}{q}\right) \phi(n)
=q\sum_{\pm}\sum_{d_1 \mid q} \sum_{d_2 = 1}^{\infty}\frac{A_\pi(d_1,d_2)}{d_1 d_2} S\left(\overline{a}, \pm d_2; \frac{q}{d_1}\right)\Phi^{\pm}\left(\frac{d_1^2d_2}{q^3}\right),
\eea
where $S(a,b;c)$ denotes the classical Kloosterman sum,
and $\Phi^{\pm} $ is the integral transform of $\phi$ given by
\bea\label{Phi}
\Phi^{\pm}\left(x\right)
=\frac{1}{2\pi i}\int_{(\sigma)}x^{-s}\gamma_{\pm}(s)\widetilde{\phi}(-s)\mathrm{d}s,\qquad
\sigma>\max\limits_{1\leq j\leq 3}\{-1-\mathrm{Re}(\alpha_j)\},
\eea
with $\widetilde{\phi}(s)=\int_0^{\infty}\phi(x)x^{s-1}\mathrm{d}x$ the Mellin transform of $\phi$ and
\bna
\gamma_{\pm}(s)=\frac{1}{2\pi^{3(s+1/2)}}
\left(\prod_{j=1}^3\frac{\Gamma\left(\frac{1+s+\alpha_j}{2}\right)} {\Gamma\left(\frac{-s-\alpha_j}{2}\right)}
\mp i\prod_{j=1}^3\frac{\Gamma\left(\frac{2+s+\alpha_j}{2}\right)} {\Gamma\left(\frac{1-s-\alpha_j}{2}\right)}\right).
\ena
Here $\alpha_j$, $j=1,2,3$ are the Langlands parameters of $\pi$.
\end{lemma}

For the integral transform $\Phi^{\pm}$ in the Voronoi summation \eqref{voronoi},
we have the following result, which, in the $\rm GL_3$ case, is a modification of Lemma 2.7 in
Jiang and L\"{u} \cite{JL}.
\begin{lemma}\label{evaluation of G}
Let $ \Phi^{\pm}(x)$ be defined as in \eqref{Phi}.
Let $\phi(x)$ be a fixed smooth function compactly supported on $[aX, bX]$
with $b>a>0$ satisfying
\bna
\phi^{(j)}(x)\ll_j \left(X/R\right)^{-j} \quad\text{ for}\; j\geq 0 \quad
\text{and} \quad
\int|\phi^{(j)}(\xi)|\mathrm{d}\xi\ll Z(X/R)^{-j+1} \quad \text{for}\; j\geq1,
\ena
for some $R, Z\geq1$. Then we have
\bna
\Phi^{\pm}(x)\ll \left\{
\begin{array}{ll}
(xX)^{-A} , & \text{if }\, x > R^{3+\varepsilon}X^{-1} ,\\
(xX)^{\frac{1}{3}}Z, & \text{if }\, X^{-1}\ll x \leq R^{3+\varepsilon}X^{-1} ,\\
(xX)^{\frac{1}{2}}ZR^\varepsilon , & \text{if }\, x \ll X^{-1},
\end{array} \right.
\ena
where the implied constants may depend on $a$, $b$.
\end{lemma}
\begin{proof}
Take $m=3$ in the proof of Lemma 2.7 in Jiang and L\"{u} \cite{JL} and take into account the extra $Z$ in the upper bound of $\int|\phi^{(j)}(\xi)|\mathrm{d}\xi$.
\end{proof}

We also employ the following result in our proof.
\begin{lemma}\label{rk}
We have
\begin{equation}\label{eq:lem-rk}
\sum_{n\le x}\bigl|A_{\pi}(a n, m)\bigr|^2 \ll_{\pi,\varepsilon} (am)^{2\vartheta+\varepsilon}\,x,
\end{equation}
where $ \vartheta= 5/14$.
\end{lemma}
\begin{proof}
Following Blomer \cite{B}, we separate the prime factors of $am$ from $n$.
We write $b\mid (am)^{\infty}$ to mean that every prime divisor of $b$ divides $am$
(i.e.\ $p\mid b\Rightarrow p\mid am$).
Then any $n$ can be written as $n=bn'$ with $b\mid (am)^{\infty}$ and $(n',am)=1$,
and hence
\begin{align}
\sum_{n\le x}\bigl|A_{\pi}(a n,m)\bigr|^2
&\le \sum_{b\mid (am)^{\infty}}\ \sum_{\substack{n'\le x/b\\ (n',abm)=1}}
\bigl|A_{\pi}(abn',m)\bigr|^2 \nonumber.
\end{align}
Then by using Hecke multiplicativity (cf.\ \cite[Theorem 6.4.11]{DG}),
\[
\sum_{n\le x}\bigl|A_{\pi}(a n,m)\bigr|^2
\ll \sum_{b\mid (am)^{\infty}} \bigl|A_{\pi}(ab,m)\bigr|^2
\sum_{n'\le x/b}\bigl|A_{\pi}(n',1)\bigr|^2.
\]
By the individual bound \eqref{kk}, we have
$|A_{\pi}(ab,m)|\ll_{\pi,\varepsilon} (abm)^{\vartheta+\varepsilon}$,
where $\vartheta=5/14$.
Moreover, the Rankin-Selberg theory (see \cite{DG}) gives
\[
\sum_{m^2n\le y}\bigl|A_{\pi}(n,m)\bigr|^2\ll_{\pi} y,
\]
which in particular implies that $\sum_{n\le y}|A_{\pi}(n,1)|^2\ll_{\pi} y$.
Finally, summing over $b\mid (am)^{\infty}$ contributes $O_{\varepsilon}((am)^{\varepsilon})$.
Combining these estimates yields \eqref{eq:lem-rk}.
\end{proof}

\section{Proof of Theorem \ref{Th11}}
\setcounter{equation}{0}

We will prove Theorem \ref{Th11}  by the classical circle method.
Let
\bna
\mathscr{S}(X)=\sum_{n_1,\cdots,n_\ell,n_{\ell+1}\in \mathbb{Z}_+
\atop n=n_1^r+\cdots+n_{\ell}^r+n_{\ell+1}^s}A_{\pi}(n,1)\omega\left(n/X \right).
\ena
For any $\alpha\in\mathbb{R}$, we define
\bea\label{FG}
\mathscr{F}_r(\alpha,X)=\sum_{n\leq X^{1/r}}e(\alpha n^r),\qquad
\mathscr{G}(\alpha,X)=\sum_{n\geq1}A_\pi(n,1)e(-\alpha n)\omega\left(n/X\right).
\eea
Then by the orthogonality relation
\begin{center}
$\int_{0}^{1} e(n\alpha)\mathrm{d}\alpha
=\begin{cases}
1,&\text{if } n=0,\\
0,&\text{if } n\in \mathbb{Z}\backslash \{0\},
\end{cases}$
\end{center}
one has
\bna
\mathscr{S}(X)=\int_{0}^{1} \mathscr{F}_r^\ell(\alpha,X)\mathscr{F}_s(\alpha,X )\mathscr{G}(\alpha,X)\mathrm{d}\alpha.
\ena
In order to apply the circle method, we choose the parameters $P$ and $Q$ such that
\bna
P=X^{\theta}, \qquad Q=X^{1-\theta},
\ena
where $\theta$ is a positive number to be decided later.
Note that $\mathscr{F}_r^\ell(\alpha, X )\mathscr{F}_s(\alpha, X )\mathscr{G}(\alpha, X)$
is a periodic function of period 1.
One further has
\bna
\mathscr{S}(X)=\int_{\frac{1}{Q}}^{1+\frac{1}{Q}} \mathscr{F}_r^\ell(\alpha, X )\mathscr{F}_s(\alpha, X )
\mathscr{G}(\alpha, X )\mathrm{d}\alpha.
\ena
By Dirichlet's lemma on rational approximation, each
$\alpha\in I:=\left[Q^{-1},1+Q^{-1}\right]$ can be written in the form
\bea\label{rational approximations}
\alpha=\frac{a}{q}+\beta,\qquad |\beta|\le\frac{1}{qQ}
\eea
for some integers $a$, $q$ with $1\le a\le q\le Q$ and $(a,q)=1$.
We denote by $ \mathfrak{M}(a,q)$ the set of $\alpha$ satisfying \eqref{rational approximations}
and define the major arcs and the minor arcs by
\bna
\mathfrak{M}=\bigcup_{1\le q\le P}\bigcup_{1\le a\le q\atop (a,q)=1}
\mathfrak{M}(a,q),\qquad
\mathfrak{m}=\Big[1/Q,1+1/Q\Big] \backslash \mathfrak{M}.
\ena
Then we have
\bea\label{deposition}
\mathscr{S}(X)=\int_\mathfrak{M} \mathscr{F}_r^\ell(\alpha, X ) \mathscr{F}_s(\alpha, X )\mathscr{G}(\alpha, X )\mathrm{d}\alpha
+\int _\mathfrak{m}\mathscr{F}_r^\ell(\alpha, X )\mathscr{F}_s(\alpha, X )\mathscr{G}(\alpha, X )\mathrm{d}\alpha.
\eea

We first consider the contribution from the minor arcs.
The proof is very similar to that in \cite[Section 2]{DS} by Du and Sun.
We include it here for completeness.
By Cauchy-Schwarz's inequality, we have
\bea\label{minorarcs00}
&&\int _\mathfrak{m}\mathscr{F}_r^\ell(\alpha, X )\mathscr{F}_s(\alpha, X )\mathscr{G}(\alpha, X )\mathrm{d}\alpha\nonumber\\
&\ll&\sup_{\alpha \in \mathfrak{m}} |\mathscr{F}_s(\alpha, X)||\mathscr{F}_r(\alpha, X)|^{\ell-2^{r-1}}
\Big(\int_{0}^{1} | \mathscr{F}_r(\alpha, X) |^{2^r}\mathrm{d}\alpha\Big)^\frac{1}{2}
\Big(\int_{0}^{1} |\mathscr{G}(\alpha, X )|^2\mathrm{d}\alpha \Big)^\frac{1}{2}.
\eea
For the first integral in \eqref{minorarcs00}, we apply Hua's lemma (see \cite[Lemma 2.5]{V}) to get
\bna\label{hua}
\int_{0}^{1}|\mathscr{F}_r(\alpha, X )|^{2^r} \mathrm{d}\alpha \ll X^{\frac{2^r}{r}-1+\varepsilon}.
\ena
For the last integral in \eqref{minorarcs00}, Lemma \ref{rk} yields
\bna\label{rkk}
\int_{0}^{1} |\mathscr{G}(\alpha, X )|^2\mathrm{d}\alpha \ll\sum_{n\ll X}|A_\pi(n,1)|^2\ll X.
\ena
Plugging these estimates into \eqref{minorarcs00} we obtain
\bea\label{minor arcs 1}
\int_\mathfrak{m}\mathscr{F}_r^\ell(\alpha,X)\mathscr{F}_s(\alpha,X)\mathscr{G}(\alpha,X)\mathrm{d}\alpha
\ll X^{\frac{2^{r-1}}{r}+\varepsilon}\sup_{\alpha \in \mathfrak{m}} |\mathscr{F}_s(\alpha, X)||\mathscr{F}_r(\alpha, X)|^{\ell-2^{r-1}}.
\eea
In order to make the upper bound as small as possible,
we distinguish four cases according to the values of $r$ and $s$.
Assume
\bea\label{theta}
\theta\leq \theta_0=\min\big\{1/r, 1/s\big\}.
\eea

\quad  (\romannumeral1) $2\le r\le 7$, $2\le s\le 7$.

In this case, we apply  Weyl's inequality (see \cite[Lemma 2.4]{V})
and use \eqref{theta} to get, for $\alpha \in \mathfrak{m}$,
\bea\label{Wely's inequality}
\mathscr{F}_r(\alpha, X) \ll X^{\frac{1}{r}+\varepsilon  }\left( P^{-1}+X^{-\frac{1}{r} } +QX^{-1}\right)^{\frac{1}{2^{r-1}}} \ll X^{\frac{1}{r}-\frac{\theta}{2^{r-1}}+\varepsilon},
\eea
and similarly,
\bea\label{F_2}
\mathscr{F}_s(\alpha, X)\ll X^{\frac{1}{s}-\frac{\theta}{2^{s-1}}+\varepsilon}.
\eea
Then by \eqref{minor arcs 1}-\eqref{F_2}, we derive
\bea\label{minor arcs}
&&\int_\mathfrak{m}\mathscr{F}_r^\ell(\alpha, X )\mathscr{F}_s(\alpha, X )\mathscr{G}(\alpha, X )\mathrm{d}\alpha\nonumber\\
&\ll&X^{\frac{2^{r-1}}{r}+\varepsilon}\cdot X^{\frac{1}{s}-\frac{\theta}{2^{s-1}}+\varepsilon}
\cdot X^{\left(\frac{1}{r}-\frac{\theta}{2^{r-1}}+\varepsilon\right)\left(\ell-2^{r-1}\right)}
\nonumber\\
&\ll& X^{\frac{\ell}{r}+\frac{1}{s}-\theta\left(\frac{\ell}{2^{r-1}}+\frac{1}{2^{s-1}}-1\right)+\varepsilon}.
\eea

\quad  (\romannumeral2) $2\le r\le 7$, $s\ge 8$.

In this case, we take Lemma 1.6 of \cite{LM} in place of Weyl's inequality, as it provides a superior result for $s\ge 8$, and get
\bea\label{Wely's inequality2}
\mathscr{F}_s(\alpha, X)\ll X^{\frac{1}{s}+\varepsilon }\left( P^{-1}+X^{-\frac{1}{s} }+QX^{-1}\right)^{\frac{1}{2s(s-1)}}
\ll X^{\frac{1}{s}-\frac{\theta}{2s(s-1)}+\varepsilon}.
\eea
Thus by \eqref{minor arcs 1}, \eqref{F_2} and \eqref{Wely's inequality2}, we have
\bea\label{minor arcs2}
&&\int _\mathfrak{m}\mathscr{F}_r^\ell(\alpha, X )\mathscr{F}_s(\alpha, X )\mathscr{G}(\alpha, X )\mathrm{d}\alpha\nonumber\\
&\ll& X^{\frac{2^{r-1}}{r}+\varepsilon}\cdot X^{\frac{1}{s}-\frac{\theta}{2s(s-1)}+\varepsilon}\cdot X^{\left(\frac{1}{r}-\frac{\theta}{2^{r-1}}+\varepsilon\right)\left(\ell-2^{r-1}\right)}\nonumber\\
&\ll&X^{\frac{\ell}{r}+\frac{1}{s}-\theta\left(\frac{\ell}{2^{r-1}}+\frac{1}{2s(s-1)}-1\right)+\varepsilon}.
\eea

\quad  (\romannumeral3) $r\ge8$, $2\le s\le 7$.

Similarly as in the case (\romannumeral2),
\bea\label{ F_r1}
\mathscr{F}_r(\alpha, X)\ll X^{\frac{1}{r}-\frac{\theta}{2r(r-1)}+\varepsilon}.
\eea
Thus by \eqref{minor arcs 1}, \eqref{F_2} and \eqref{ F_r1}, one has
\bea\label{romannumeral3}
&&\int _\mathfrak{m}\mathscr{F}_r^\ell(\alpha, X )\mathscr{F}_s(\alpha, X )\mathscr{G}(\alpha, X )\mathrm{d}\alpha\nonumber\\
&\ll& X^{\frac{2^{r-1}}{r}+\varepsilon}\cdot X^{\frac{1}{s}-\frac{\theta}{2^{s-1}}+\varepsilon}\cdot X^{\left(\frac{1}{r}-\frac{\theta}{2r(r-1)}+\varepsilon\right)\left(\ell-2^{r-1}\right)}\nonumber\\
&\ll&X^{\frac{\ell}{r}+\frac{1}{s}-\theta\left(\frac{\ell-2^{r-1}}{2r(r-1)}+\frac{1}{2^{s-1}}\right)+\varepsilon}.
\eea

\quad  (\romannumeral4) $r\ge8$, $s\ge 8$.

By \eqref{minor arcs 1}, \eqref{Wely's inequality2} and \eqref{ F_r1}, one has
\bea\label{romannumeral4}
&&\int _\mathfrak{m}\mathscr{F}_r^\ell(\alpha, X )\mathscr{F}_s(\alpha, X )\mathscr{G}(\alpha, X )\mathrm{d}\alpha\nonumber\\
&\ll& X^{\frac{2^{r-1}}{r}+\varepsilon}\cdot X^{\frac{1}{s}-\frac{\theta}{2s(s-1)}+\varepsilon} \cdot X^{\left(\frac{1}{r}-\frac{\theta}{2r(r-1)}+\varepsilon\right)\left(\ell-2^{r-1}\right)}\nonumber\\
&\ll&X^{\frac{\ell}{r}+\frac{1}{s}-\theta\left(\frac{\ell-2^{r-1}}{2r(r-1)}+\frac{1}{2s(s-1)}\right)+\varepsilon}.
\eea
This finishes the treatment of the minor arcs.
The integral over the major arcs will be handled in the next section,
and the proof of Theorem \ref{Th11} will be completed in the last section.

\section{ The integral over the major arcs}\label{majorr}

By the definition of the major arcs,  we have
\bea\label{major1}
&&\int_\mathfrak{M}\mathscr{F}_r^\ell(\alpha, X )\mathscr{F}_s(\alpha, X )\mathscr{G}(\alpha, X )\mathrm{d}\alpha\nonumber\\
&=&\sum_{q\le P}\;\sideset{}{^*}\sum_{a \bmod q}\int_{\mathfrak{M} (a,q)}\mathscr{F}_r^\ell(\alpha, X )\mathscr{F}_s(\alpha, X )\mathscr{G}(\alpha, X )\mathrm{d}\alpha\nonumber\\
&=&\sum_{q\le P}\int_{|\beta |\le \frac{1}{qQ}} \;\sideset{}{^*}\sum_{a \bmod q}\mathscr{F}_r^\ell\left(\frac{a}{q}+\beta, X\right) \mathscr{F}_s\left(\frac{a}{q}+\beta, X\right)\mathscr{G}\left(\frac{a}{q}+\beta, X\right)\mathrm{d}\beta,
\eea
where, throughout the paper the $*$ denotes the condition $(a,q)=1$.

For an asymptotic formula of $\mathscr{F}_r\left(a/q+\beta,X\right)$,
we quote the following result (see \cite[Theorem 4.1]{V}).
\begin{lemma}\label{F}
Let $(a, q)=1$, and $|\beta| \leq 1/(qQ)$. We have
\bna
\mathscr{F}_r\left(\frac{a}{q}+\beta, X\right)=\frac{G_r(a,0;q)} q\Psi_r(\beta) +O\left(q^{\frac{1}{2}+\varepsilon}\left(1+ |\beta|X\right)^\frac{1}{2}\right),
\ena
where $G_r(a,0;q)$ is the Gauss sum
\bea\label{Gauss sum}
G_r(a,b;q)=\sum_{x\bmod q}e\Big(\frac{ax^r+bx}{q}\Big)
\eea
and
\bea\label{Psi0}
\Psi_r(\beta)=\int_{0}^{X^{1/r}} e\left(\beta u^{r} \right) \mathrm{d} u.
\eea
\end{lemma}

By the $r$-th derivative test and the trivial estimate, one has
\bea\label{Psi}
\Psi_r(\beta)\ll\left(\frac{X}{1+|\beta|X} \right)^{1/r}.
\eea
By Lemma \ref{F}, we have
\bna
\mathscr{F}_r\left(\frac{a}{q}+\beta, X\right)=\frac{G_r(a,0;q)} q\Psi_r(\beta)+E_r(q,\beta),
\ena
where
\bea\label{Er}
E_r(q,\beta)\ll q^{\frac{1}{2}+\varepsilon}\left(1+|\beta|X\right)^\frac{1}{2}.
\eea
Thus
\bea\label{Fr}
\mathscr{F}_r^\ell\left(\frac{a}{q}+\beta, X\right)=\sum_{i=0}^{\ell}\binom{\ell}{i} \frac{G_r^{\ell-i}(a,0;q)} {q^{\ell-i}}\Psi_r^{\ell-i}(\beta)E_r^i(q,\beta).
\eea

For $\mathscr{G}(\alpha,X)$ in \eqref{FG}, we apply Lemma \ref{gl3voronoi}
with $\phi_{\beta}(x)=\omega\left(x/X\right)e(-\beta x)$  to obtain
\bea\label{eq:G}
\mathscr{G}\left(\frac{a}{q}+\beta,X\right)
&=&\sum_{n\geq 1}A_\pi(n,1)e\left(-\frac{a n}{q}\right)\phi_{\beta}(n)\nonumber\\
&=&q\sum_{\pm}\sum_{d_1 \mid q} \sum_{d_2 = 1}^{\infty} \frac{A_\pi(d_1,d_2)}{d_1 d_2} S\left(-\overline{a}, \pm d_2; \frac{q}{d_1}\right) \Phi_{\beta}^{\pm}\left(\frac{d_1^2d_2}{q^3 }\right),
\eea
where by \eqref{Phi}.
\bna
\Phi_{\beta}^{\pm}\left(x\right)=\frac{1}{2\pi i}\int_{(\sigma)}x^{-s}\gamma_{\pm}(s)\widetilde{\phi_{\beta}}(-s)\mathrm{d}s.
\ena
By \eqref{Fr} and \eqref{eq:G}, we have
\bea\label{d2}
&&\sideset{}{^*}\sum_{a \bmod q}\mathscr{F}_r^\ell\left(\frac{a}{q}+\beta, X\right)\mathscr{F}_s\left(\frac{a}{q}+\beta, X\right)\mathscr{G}\left(\frac{a}{q}+\beta, X\right)\nonumber\\
&=&\sideset{}{^*}\sum_{a \bmod q}\left\{\sum_{i=0}^{\ell}\binom{\ell}{i} \frac{G_r^{\ell-i}(a,0;q)}{q^{\ell-i}}\Psi_r^{\ell-i}(\beta)E_r^{i}(q,\beta)\right\} \left\{\frac{G_s(a,0;q)} q\Psi_s(\beta)+E_s(q,\beta)\right\}\nonumber\\
&&\left\{q\sum_{\pm}\sum_{d_1 \mid q} \sum_{d_2 = 1}^{\infty}\frac{A_\pi(d_1,d_2)}{d_1 d_2} S\left(-\overline{a}, \pm d_2; \frac{q}{d_1}\right)\Phi_{\beta}^{\pm}\left(\frac{d_1^2d_2}{q^3}\right)\right\}\nonumber\\
&=&\sum_{i=0}^{\ell}\binom{\ell}{i}\sum_{\pm}\sum_{d_1 \mid q}\frac{\Psi_r^{\ell-i}(\beta)E_r^{i}(q,\beta)\Psi_s(\beta)}{q^{\ell-i}}\sum_{d_2 = 1}^{\infty}\frac{A_\pi(d_1,d_2)}{d_1 d_2}\Phi_{\beta}^{\pm}\left(\frac{d_1^2d_2}{q^3}\right)\mathfrak{C}_{1,i}(d_1,\pm d_2;q)\nonumber\\
&&+q\sum_{i=0}^{\ell}\binom{\ell}{i}\sum_{\pm}\sum_{d_1 \mid q}\frac{\Psi_r^{\ell-i}(\beta)E_r^{i}(q,\beta)E_s(q,\beta)}{q^{\ell-i}}\sum_{d_2 = 1}^{\infty}\frac{A_\pi(d_1,d_2)}{d_1 d_2}\Phi_{\beta}^{\pm}\left(\frac{d_1^2d_2}{q^3}\right)\mathfrak{C}_{2,i}(d_1,\pm d_2;q),\nonumber\\
\eea
where
\bna\label{Ri}
&&\mathfrak{C}_{1,i}(d_1,d_2;q)=\sideset{}{^*}\sum_{a \bmod q}G_r^{\ell-i}(a,0;q)G_s(a,0;q)S\left(-\overline{a},  d_2; \frac{q}{d_1}\right), \\
&&\mathfrak{C}_{2,i}(d_1,d_2;q)=\sideset{}{^*}\sum_{a \bmod q}G_r^{\ell-i}(a,0;q)S\left(-\overline{a}, d_2; \frac{q}{d_1}\right).
\ena

By Theorem 4.2 in \cite{V}, we have
\bna
G_r(a,0;q) \ll q^{1-\frac{1}{r}+\varepsilon}.
\ena
Combining this estimate with Weil's bound for Kloosterman sum, we obtain
\bea\label{Ci}
&&\mathfrak{C}_{1,i}(d_1,d_2;q)\ll
q^{\left(1-\frac{1}{r}\right)(\ell-i)+2-\frac{1}{s}+\varepsilon}\Big(d_2,\frac{q}{d_1}\Big)^{1/2}\left(\frac{q}{d_1}\right)^{1/2}, \label{C1}\\
&&\mathfrak{C}_{2,i}(d_1,d_2;q)\ll q^{\left(1-\frac{1}{r}\right)(\ell-i)+1+\varepsilon} \Big(d_2,\frac{q}{d_1}\Big)^{1/2}\left(\frac{q}{d_1}\right)^{1/2}.\label{C2}
\eea

For the two sums involving the integral $\Phi_{\beta}^{\pm}(x)$ in \eqref{d2},
we have the following estimates.
\begin{proposition}\label{pro1}
For any $\varepsilon>0$, we have
\bna
\sum_{d_2=1}^{\infty}\frac{|A_{\pi}(d_1,d_2)|}{d_1d_2}\Big|\Phi_{\beta}^{\pm}\left(\frac{d_1^2d_2}{q^3}\right)\Big|\Big(d_2,\frac{q}{d_1}\Big)^{1/2}
\ll_{\pi,\varepsilon}d_1^{\vartheta-1}(1+|\beta|X)(\Delta+|\beta|X)X^\varepsilon,
\ena
where $\vartheta=5/14$.
\end{proposition}
\begin{proof}
Recall that $\phi_{\beta}(x)=\omega\left(x/X\right)e(-\beta x)$,
where $\omega(x)$ is a smooth function compactly supported in $[1, 2]$  and satisfies
\bna
\omega^{(j)}(x)\ll_j \Delta^j \quad\text{ for}\quad j\geq 0 \quad
\text{and} \quad
\int|\omega^{(j)}(\xi)|\mathrm{d}\xi\ll\Delta^{j-1} \quad \text{for}\quad j\geq1.
\ena
Thus we have
\bna
\phi_\beta^{(j)}(x)\ll\left(\frac{X}{\Delta+|\beta|X}\right)^{-j} \quad\text{ for}\; j\geq 0,
\ena
and
\bna
\int_{0}^{\infty}|\phi_{\beta}^{(j)}(x)|\mathrm{d}x
&=&\int_{0}^{\infty}\bigg|\sum_{0\leq i\leq j}\binom{j}{i}(-2\pi i \beta)^{j-i}e(-\beta x)X^{-i}\omega^{(i)}\left(\frac{x}{X}\right)\bigg|\mathrm{d}x\\
&\ll&|\beta|^jX+\sum_{1\leq i\leq j}|\beta|^{j-i}X^{-i+1}\int_{0}^{\infty}|\omega^{(i)}(\xi)|\mathrm{d}\xi\\
&\ll&|\beta|^{j}X+|\beta|^{j-1}\sum_{1\leq i\leq j}\left(\frac{\Delta}{|\beta|X}\right)^{i-1}\\
&\ll&(1+|\beta|X)\left(\frac{X}{\Delta+|\beta|X}\right)^{-j+1}
\ena
for any $j\geq 1$.
Hence, applying Lemma \ref{evaluation of G} with
$Z=1+|\beta|X$ and $R=\Delta+|\beta|X$, we obtain
\bea\label{eq:integral}
\Phi_\beta^{\pm}(x)\ll
\left\{
\begin{array}{ll}
(xX)^{-A} , & \text{if }\, x > (\Delta+|\beta|X)^{3+\varepsilon}X^{-1} ,\\
(xX)^{\frac{1}{3}}(1+|\beta|X), & \text{if }\, X^{-1}\ll x \leq (\Delta+|\beta|X)^{3+\varepsilon}X^{-1} ,\\
(xX)^{\frac{1}{2}}(1+|\beta|X)(\Delta+|\beta|X)^\varepsilon, & \text{if }\, x \ll X^{-1}.
\end{array} \right.
\eea
For any $\varepsilon>0$, we can choose $A$ sufficiently large to deduce that $\Phi_{\beta}^{\pm}\left(d_1^2d_2/q^3\right)$
is negligible unless $d_1^2d_2/q^3\ll (\Delta+|\beta|X)^{3+\varepsilon}X^{-1}$.
Therefore, applying dyadic subdivision to the sum over $d_2$ and using \eqref{eq:integral}, one has
\bea\label{eq:A sum}
&&\sum_{d_2=1}^{\infty}\frac{|A_{\pi}(d_1,d_2)|}{d_1d_2}\Big|\Phi_{\beta}^{\pm}\left(\frac{d_1^2d_2}{q^3}\right)\Big|\Big(d_2,\frac{q}{d_1}\Big)^{1/2}\nonumber\\
&\ll& X^\varepsilon \max_{ q^3d_1^{-2}X^{-1}\ll D_1\ll q^3d_1^{-2}X^{-1}(\Delta+|\beta|X)^{3+\varepsilon}}\;\sum_{d_2 \sim D_1}\frac{|A_{\pi}(d_1,d_2)|}{d_1d_2}\left(\frac{d_1^2d_2}{q^3}X\right)^{1/3}(1+|\beta|X)\Big(d_2,\frac{q}{d_1}\Big)^{1/2}\nonumber\\
&&+X^\varepsilon \max_{ D_2\ll q^3d_1^{-2}X^{-1}}\;\sum_{d_2 \sim D_2}\frac{|A_{\pi}(d_1,d_2)|}{d_1d_2}\left(\frac{d_1^2d_2}{q^3}X\right)^{1/2}(1+|\beta|X)\Big(d_2,\frac{q}{d_1}\Big)^{1/2}\nonumber\\
&\ll& \frac{X^{1/3+\varepsilon}(1+|\beta|X)}{d_1^{1/3}q}\max_{ q^3d_1^{-2}X^{-1}\ll D_1\ll q^3d_1^{-2}X^{-1}(\Delta+|\beta|X)^{3+\varepsilon}}\;\sum_{d_2 \sim D_1}\frac{|A_{\pi}(d_1,d_2)|}{d_2^{2/3}}\Big(d_2,\frac{q}{d_1}\Big)^{1/2}\nonumber\\
&&+\frac{X^{1/2+\varepsilon}(1+|\beta|X)}{q^{3/2}} \max_{ D_2\ll q^3d_1^{-2}X^{-1}}\;\sum_{d_2 \sim D_2}\frac{|A_{\pi}(d_1,d_2)|}{d_2^{1/2}}\Big(d_2,\frac{q}{d_1}\Big)^{1/2}\nonumber\\
&\ll& \frac{X^{1/3+\varepsilon}(1+|\beta|X)}{d_1^{1/3}q}\max_{ q^3d_1^{-2}X^{-1}\ll D_1\ll q^3d_1^{-2}X^{-1}(\Delta+|\beta|X)^{3+\varepsilon}}\;\sum_{\ell |qd_1^{-1}}\ell^{-1/6}\;\sum_{d_2\sim D_1\ell^{-1}}\frac{|A_{\pi}(d_1,\ell d_2)|}{d_2^{2/3}}\nonumber\\
&&+\frac{X^{1/2+\varepsilon}(1+|\beta|X)}{q^{3/2}}\max_{ D_2\ll q^3d_1^{-2}X^{-1}}\;\sum_{\ell |qd_1^{-1}}\sum_{d_2 \sim D_2\ell^{-1}}\frac{|A_{\pi}(d_1,\ell d_2)|}{d_2^{1/2}}.
\eea

By Cauchy-Schwarz's inequality and Lemma \ref{rk}, we have
\bea\label{eq:1}
&&\sum_{\ell |qd_1^{-1}}\ell^{-1/6}\;\sum_{d_2\sim D_1\ell^{-1}}\frac{|A_{\pi}(d_1,\ell d_2)|}{d_2^{2/3}}\nonumber\\
&\ll& \sum_{\ell |qd_1^{-1}}\ell^{-1/6}\left(\frac{D_1}{\ell}\right)^{-2/3}\left(\sum_{d_2\sim D_1\ell^{-1}}1\right)^{1/2}\left(\sum_{d_2\sim D_1\ell^{-1}}|A_\pi(d_1,\ell d_2)|^2\right)^{1/2}\nonumber\\
&\ll& \sum_{\ell |qd_1^{-1}}\ell^{-\frac{1}{6}}\left(\frac{D_1}{\ell}\right)^{-2/3}\left(\frac{D_1}{\ell}\right)^{1/2}\left((d_1\ell)^{2\vartheta+\varepsilon}\frac{D_1}{\ell}\right)^{1/2}\nonumber\\
&\ll& d_1^{\vartheta+\varepsilon}D_1^{1/3}\sum_{\ell |qd_1^{-1}}\ell^{-1/2+\vartheta+\varepsilon}\nonumber\\
&\ll&X^\varepsilon d_1^\vartheta D_1^{1/3},
\eea
where $\vartheta=5/14$.
Similarly,
\bea\label{eq:22}
&&\sum_{\ell |qd_1^{-1}}\sum_{d_2\sim D_2\ell^{-1}}\frac{|A_{\pi}(d_1,\ell d_2)|}{d_2^{1/2}}\nonumber\\
&\ll& \sum_{\ell |qd_1^{-1}}\left(\frac{D_2}{\ell}\right)^{-1/2}\left(\sum_{d_2\sim D_2\ell^{-1}}1\right)^{1/2}\left(\sum_{d_2\sim D_2\ell^{-1}}|A_\pi(d_1,\ell d_2)|^2\right)^{1/2}\nonumber\\
&\ll& \sum_{\ell |qd_1^{-1}}\left(\frac{D_2}{\ell}\right)^{-1/2}\left(\frac{D_2}{\ell}\right)^{1/2}\left((d_1\ell)^{2\vartheta+\varepsilon}\frac{D_2}{\ell}\right)^{1/2}\nonumber\\
&\ll& d_1^{\vartheta+\varepsilon}D_2^{1/2}\sum_{\ell |qd_1^{-1}}\ell^{-1/2+\vartheta+\varepsilon}\nonumber\\
&\ll&X^\varepsilon d_1^\vartheta D_2^{1/2}.
\eea
Inserting \eqref{eq:1} and \eqref{eq:22} into \eqref{eq:A sum}, we obtain
\bna
&&\sum_{d_2=1}^{\infty}\frac{|A_{\pi}(d_1,d_2)|}{d_1d_2}\Big|\Phi_{\beta}^{\pm}\left(\frac{d_1^2d_2}{q^3}\right)\Big|\Big(d_2,\frac{q}{d_1}\Big)^{1/2}\nonumber\\
&\ll& \frac{X^{1/3+\varepsilon}(1+|\beta|X)}{d_1^{1/3}q}\max_{ q^3d_1^{-2}X^{-1}\ll D_1\ll q^3d_1^{-2}X^{-1}(\Delta+|\beta|X)^{3+\varepsilon}}\;d_1^\vartheta D_1^{1/3}\nonumber\\
&&+\frac{X^{1/2+\varepsilon}(1+|\beta|X)}{q^{3/2}} \max_{ D_2\ll q^3d_1^{-2}X^{-1}}\;d_1^\vartheta D_2^{1/2}\nonumber\\
&\ll& \frac{X^{1/3+\varepsilon}(1+|\beta|X)}{d_1^{1/3}q}d_1^\vartheta \left(q^3d_1^{-2}X^{-1}(\Delta+|\beta|X)^{3+\varepsilon}\right)^{1/3}\nonumber\\
&&+\frac{X^{1/2+\varepsilon}(1+|\beta|X)}{q^{3/2}}d_1^\vartheta \left(q^3d_1^{-2}X^{-1}\right)^{1/2}\nonumber\\
&\ll&d_1^{\vartheta-1}(1+|\beta|X)(\Delta+|\beta|X)X^\varepsilon+d_1^{\vartheta-1}(1+|\beta|X)X^\varepsilon\nonumber\\
&\ll&d_1^{\vartheta-1}(1+|\beta|X)(\Delta+|\beta|X)X^\varepsilon,
\ena
where $\vartheta=5/14$.
This completes the proof of the proposition.
\end{proof}

Now we return to the estimation of the integral over the major arcs.
Inserting \eqref{d2} into \eqref{major1}, we have
\bea\label{med-estimate}
\int_\mathfrak{M}\mathscr{F}_r^\ell(\alpha, X )\mathscr{F}_s(\alpha, X)\mathscr{G}(\alpha, X )\mathrm{d}\alpha
=\sum_{\pm}\sum_{i=0}^{\ell}\binom{\ell}{i}\mathbf{M}_i^{\pm}+\sum_{\pm}\sum_{i=0}^{\ell}\binom{\ell}{i}\mathbf{R}_i^{\pm},
\eea
where
\bea\label{MMM}
\mathbf{M}_i^{\pm}=\sum_{q\le P}\int_{|\beta |\le \frac{1}{qQ}} \;\sum_{d_1 \mid q}\frac{\Psi_r^{\ell-i}(\beta)E_r^{i}(q,\beta)\Psi_s(\beta)}{q^{\ell-i}}\sum_{d_2 = 1}^{\infty}\frac{A_\pi(d_1,d_2)}{d_1 d_2}\Phi_{\beta}^{\pm}\left(\frac{d_1^2d_2}{q^3}\right)\mathfrak{C}_{1,i}(d_1,\pm d_2;q)\mathrm{d}\beta,%\nonumber\\
\eea
and
\bea\label{RRR}
\mathbf{R}_i^{\pm} =\sum_{q\le P}q\int_{|\beta |\le \frac{1}{qQ}} \;\sum_{d_1 \mid q}\frac{\Psi_r^{\ell-i}(\beta)E_r^{i}(q,\beta)E_s(q,\beta)}{q^{\ell-i}}\sum_{d_2 = 1}^{\infty}\frac{A_\pi(d_1,d_2)}{d_1 d_2}\Phi_{\beta}^{\pm}\left(\frac{d_1^2d_2}{q^3}\right)\mathfrak{C}_{2,i}(d_1,\pm d_2;q)\mathrm{d}\beta.%\nonumber\\
\eea

By \eqref{Psi}, \eqref{Er}, \eqref{C1} and Proposition \ref{pro1},
$\mathbf{M}_i^{\pm}$ in \eqref{MMM} is bounded by
\bea\label{M-estimate}
&&X^\varepsilon\sum_{q\le P}q^{-\frac{\ell-i}{r}-\frac{1}{s}
+\frac{5}{2}}\sum_{d_1 \mid q}d_1^{\vartheta-3/2}\int_{|\beta |\le \frac{1}{qQ}} \;\left(\frac{X}{1+|\beta|X} \right)^{\frac{\ell-i}{r}+\frac{1}{s}}q^{\frac{i}{2}}(1+|\beta|X)^{\frac{i}{2}+1}(\Delta+|\beta|X)\mathrm{d}\beta\nonumber\\
&\ll&X^{\frac{\ell-i}{r}+\frac{1}{s}+\varepsilon}\sum_{q\le P}q^{-\frac{\ell-i}{r}-\frac{1}{s}+\frac{5+i}{2}}\int_{|\beta |\le \frac{1}{qQ}} \;\left(1+|\beta|X\right)^{1+\frac{i}{2}-\frac{\ell-i}{r}-\frac{1}{s}}(\Delta+|\beta|X)\mathrm{d}\beta\nonumber\\
&\ll&X^{\frac{\ell-i}{r}+\frac{1}{s}+\varepsilon}\sum_{q\le P}q^{-\frac{\ell-i}{r}-\frac{1}{s}+\frac{5+i}{2}}\int_{|\beta |\le \frac{1}{X}} \;\Delta \mathrm{d}\beta\nonumber\\
&&+\Delta X^{1+\frac{i}{2}+\varepsilon}\sum_{q\le P}q^{-\frac{\ell-i}{r}-\frac{1}{s}+\frac{5+i}{2}}\int_{\frac{1}{X}<|\beta |\le \frac{1}{qQ}} \;|\beta|^{1+\frac{i}{2}-\frac{\ell-i}{r}-\frac{1}{s}}\mathrm{d}\beta\nonumber\\
&&+X^{2+\frac{i}{2}+\varepsilon}\sum_{q\le P}q^{-\frac{\ell-i}{r}-\frac{1}{s}+\frac{5+i}{2}}\int_{\frac{1}{X}<|\beta |\le \frac{1}{qQ}} \;|\beta|^{2+\frac{i}{2}-\frac{\ell-i}{r}-\frac{1}{s}}\mathrm{d}\beta,
\eea
where $\vartheta=5/14$.
Recall that $P=X^\theta$ and $PQ=X$. The first term in \eqref{M-estimate} can be estimated as
\bea\label{first}
&&X^{\frac{\ell-i}{r}+\frac{1}{s}+\varepsilon}\sum_{q\le P}q^{-\frac{\ell-i}{r}-\frac{1}{s}+\frac{5+i}{2}}\int_{|\beta |\le \frac{1}{X}} \;\Delta \mathrm{d}\beta\nonumber\\
&\ll&
\begin{cases}
\Delta X^{\frac{\ell-i}{r}+\frac{1}{s}-1+\varepsilon},
& \text{ if } \frac{\ell}{r}\ge \frac{7}{2}+i\left( \frac{1}{2}+\frac{1}{r}\right)-\frac{1}{s}, \\
\\
\Delta X^{\frac{\ell}{r}+\frac{1}{s}-1-\left(\frac{\ell}{r}+\frac{1}{s}-\frac{7}{2}\right)\theta-\left(\frac{1}{r}-(\frac{1}{2}+\frac{1}{r})\theta\right)i+\varepsilon},
& \text{ if } \frac{\ell}{r}< \frac{7}{2}+i\left( \frac{1}{2}+\frac{1}{r}\right)-\frac{1}{s}.
\end{cases}
\eea
Similarly, the second term in \eqref{M-estimate} can be estimated as
\bea\label{eq:2}
&&\Delta X^{1+\frac{i}{2}+\varepsilon}\sum_{q\le P}q^{-\frac{\ell-i}{r}-\frac{1}{s}+\frac{5+i}{2}}\int_{\frac{1}{X}<|\beta |\le \frac{1}{qQ}} \;|\beta|^{1+\frac{i}{2}-\frac{\ell-i}{r}-\frac{1}{s}}\mathrm{d}\beta\nonumber\\
&\ll&\Delta X^{1+\frac{i}{2}+\varepsilon}\sum_{q\le P}q^{-\frac{\ell-i}{r}-\frac{1}{s}+\frac{5+i}{2}}
\begin{cases}
\left(\frac{1}{X}\right)^{2+\frac{i}{2}-\frac{\ell-i}{r}-\frac{1}{s}},
& \text{ if } \frac{\ell}{r}\ge 2+i\left( \frac{1}{2}+\frac{1}{r}\right)-\frac{1}{s},\nonumber\\
\\
\left(\frac{1}{qQ}\right)^{2+\frac{i}{2}-\frac{\ell-i}{r}-\frac{1}{s}},
& \text{ if } \frac{\ell}{r}< 2+i\left( \frac{1}{2}+\frac{1}{r}\right)-\frac{1}{s},
\end{cases}\nonumber\\
&\ll&
\begin{cases}
\Delta X^{\frac{\ell-i}{r}+\frac{1}{s}-1+\varepsilon},
& \text{ if } \frac{\ell}{r}\ge \frac{7}{2}+i\left( \frac{1}{2}+\frac{1}{r}\right)-\frac{1}{s},\nonumber\\
\\
\Delta X^{\frac{\ell}{r}+\frac{1}{s}-1-\left(\frac{\ell}{r}+\frac{1}{s}-\frac{7}{2}\right)\theta-\left(\frac{1}{r}-(\frac{1}{2}+\frac{1}{r})\theta\right)i+\varepsilon},
& \text{ if } 2+i\left( \frac{1}{2}+\frac{1}{r}\right)-\frac{1}{s}\le \frac{\ell}{r}< \frac{7}{2}+i\left( \frac{1}{2}+\frac{1}{r}\right)-\frac{1}{s},\nonumber\\
\\
\Delta X^{\frac{\ell}{r}+\frac{1}{s}-1-\left(\frac{\ell}{r}+\frac{1}{s}-\frac{7}{2}\right)\theta-\left(\frac{1}{r}-(\frac{1}{2}+\frac{1}{r})\theta\right)i+\varepsilon},
& \text{ if } \frac{\ell}{r}< 2+i\left( \frac{1}{2}+\frac{1}{r}\right)-\frac{1}{s},
\end{cases}\nonumber\\
&\ll&
\begin{cases}
\Delta X^{\frac{\ell-i}{r}+\frac{1}{s}-1+\varepsilon},
&\text{ if } \frac{\ell}{r}\ge \frac{7}{2}+i\left( \frac{1}{2}+\frac{1}{r}\right)-\frac{1}{s}, \\
\\
\Delta X^{\frac{\ell}{r}+\frac{1}{s}-1-\left(\frac{\ell}{r}+\frac{1}{s}-\frac{7}{2}\right)\theta-\left(\frac{1}{r}-(\frac{1}{2}+\frac{1}{r})\theta\right)i+\varepsilon},
& \text{ if } \frac{\ell}{r}< \frac{7}{2}+i\left( \frac{1}{2}+\frac{1}{r}\right)-\frac{1}{s},
\end{cases}
\eea
and the last term in \eqref{M-estimate} can be estimated as
\bea\label{eq:3}
&&X^{2+\frac{i}{2}+\varepsilon}\sum_{q\le P}q^{-\frac{\ell-i}{r}-\frac{1}{s}+\frac{5+i}{2}}\int_{\frac{1}{X}<|\beta |\le \frac{1}{qQ}} \;|\beta|^{2+\frac{i}{2}-\frac{\ell-i}{r}-\frac{1}{s}}\mathrm{d}\beta\nonumber\\
&\ll&X^{2+\frac{i}{2}+\varepsilon}\sum_{q\le P}q^{-\frac{\ell-i}{r}-\frac{1}{s}+\frac{5+i}{2}}
\begin{cases}
\left(\frac{1}{X}\right)^{3+\frac{i}{2}-\frac{\ell-i}{r}-\frac{1}{s}},
& \text{ if } \frac{\ell}{r}\ge 3+i\left( \frac{1}{2}+\frac{1}{r}\right)-\frac{1}{s}, \\
\\
\left(\frac{1}{qQ}\right)^{3+\frac{i}{2}-\frac{\ell-i}{r}-\frac{1}{s}},
& \text{ if } \frac{\ell}{r}< 3+i\left( \frac{1}{2}+\frac{1}{r}\right)-\frac{1}{s},
\end{cases}\nonumber\\
&\ll&
\begin{cases}
X^{\frac{\ell-i}{r}+\frac{1}{s}-1+\varepsilon},
& \text{ if } \frac{\ell}{r}\ge \frac{7}{2}+i\left( \frac{1}{2}+\frac{1}{r}\right)-\frac{1}{s}, \\
\\
X^{\frac{\ell}{r}+\frac{1}{s}-1-\left(\frac{\ell}{r}+\frac{1}{s}-\frac{7}{2}\right)\theta-\left(\frac{1}{r}-(\frac{1}{2}+\frac{1}{r})\theta\right)i+\varepsilon},
& \text{ if } 3+i\left( \frac{1}{2}+\frac{1}{r}\right)-\frac{1}{s}\le \frac{\ell}{r}< \frac{7}{2}+i\left( \frac{1}{2}+\frac{1}{r}\right)-\frac{1}{s},\nonumber\\
\\
X^{\frac{\ell}{r}+\frac{1}{s}-1-\left(\frac{\ell}{r}+\frac{1}{s}-\frac{7}{2}\right)\theta-\left(\frac{1}{r}-(\frac{1}{2}+\frac{1}{r})\theta\right)i+\varepsilon},
& \text{ if } \frac{\ell}{r}< 3+i\left( \frac{1}{2}+\frac{1}{r}\right)-\frac{1}{s},
\end{cases}\\
&\ll&
\begin{cases}
X^{\frac{\ell-i}{r}+\frac{1}{s}-1+\varepsilon},
& \text{ if } \frac{\ell}{r}\ge \frac{7}{2}+i\left( \frac{1}{2}+\frac{1}{r}\right)-\frac{1}{s}, \\
\\
X^{\frac{\ell}{r}+\frac{1}{s}-1-\left(\frac{\ell}{r}+\frac{1}{s}-\frac{7}{2}\right)\theta-\left(\frac{1}{r}-(\frac{1}{2}+\frac{1}{r})\theta\right)i+\varepsilon},
& \text{ if } \frac{\ell}{r}< \frac{7}{2}+i\left( \frac{1}{2}+\frac{1}{r}\right)-\frac{1}{s}.
\end{cases}
\eea
By \eqref{M-estimate}-\eqref{eq:3}, we conclude that
\bea\label{M-estimate2}
\mathbf{M}_i^{\pm}\ll\begin{cases}
\Delta X^{\frac{\ell}{r}+\frac{1}{s}-1-\left(\frac{\ell}{r}+\frac{1}{s}-\frac{7}{2}\right)\theta-\left(\frac{1}{r}-(\frac{1}{2}+\frac{1}{r})\theta\right)i+\varepsilon},
& \text{ if } \frac{\ell}{r}< \frac{7}{2}+i\left( \frac{1}{2}+\frac{1}{r}\right)-\frac{1}{s}, \\
\\
\Delta  X^{\frac{\ell-i}{r}+\frac{1}{s}-1+\varepsilon},
& \text{ if } \frac{\ell}{r}\ge \frac{7}{2}+i\left( \frac{1}{2}+\frac{1}{r}\right)-\frac{1}{s}.
\end{cases}
\eea

By \eqref{Psi}, \eqref{Er}, \eqref{C2} and Proposition \ref{pro1},
$\mathbf{R}_i^{\pm}$ in \eqref{RRR} is at most
\bea\label{R-estimate}
&&X^\varepsilon\sum_{q\le P}q^{-\frac{\ell-i}{r}+\frac{5}{2}}\sum_{d_1 \mid q}d_1^{\vartheta-3/2}\int_{|\beta |\le \frac{1}{qQ}} \;\left(\frac{X}{1+|\beta|X} \right)^{\frac{\ell-i}{r}}q^{\frac{i+1}{2}}(1+|\beta|X)^{\frac{3+i}{2}}(\Delta+|\beta|X)\mathrm{d}\beta\nonumber\\
&\ll&X^{\frac{\ell-i}{r}+\varepsilon}\sum_{q\le P}q^{-\frac{\ell-i}{r}+\frac{i}{2}+3}\int_{|\beta |\le \frac{1}{qQ}} \;\left(1+|\beta|X\right)^{\frac{3+i}{2}-\frac{\ell-i}{r}}(\Delta+|\beta|X) \mathrm{d}\beta\nonumber\\
&\ll&X^{\frac{\ell-i}{r}+\varepsilon}\sum_{q\le P}q^{-\frac{\ell-i}{r}+\frac{i}{2}+3}\int_{|\beta |\le \frac{1}{X}} \;\Delta \mathrm{d}\beta\nonumber\\
&&+\Delta X^{\frac{3+i}{2}+\varepsilon}\sum_{q\le P}q^{-\frac{\ell-i}{r}+\frac{i}{2}+3}\int_{\frac{1}{X}<|\beta |\le \frac{1}{qQ}} \;|\beta|^{\frac{3+i}{2}-\frac{\ell-i}{r}}\mathrm{d}\beta\nonumber\\
&&+X^{\frac{5+i}{2}+\varepsilon}\sum_{q\le P}q^{-\frac{\ell-i}{r}+\frac{i}{2}+3}\int_{\frac{1}{X}<|\beta |\le \frac{1}{qQ}} \;|\beta|^{\frac{5+i}{2}-\frac{\ell-i}{r}}\mathrm{d}\beta,
\eea
where $\vartheta=5/14$.
The first term in \eqref{R-estimate} can be estimated as
\bea\label{Rfirst}
&&X^{\frac{\ell-i}{r}+\varepsilon}\sum_{q\le P}q^{-\frac{\ell-i}{r}+\frac{i}{2}+3}\int_{|\beta |\le \frac{1}{X}} \;\Delta \mathrm{d}\beta\nonumber\\
&\ll&
\begin{cases}
\Delta X^{\frac{\ell-i}{r}-1+\varepsilon},
& \text{ if } \frac{\ell}{r}\ge 4+i\left( \frac{1}{2}+\frac{1}{r}\right), \\
\\
\Delta X^{\frac{\ell}{r}-1-\left(\frac{\ell}{r}-4\right)\theta-\left(\frac{1}{r}-(\frac{1}{2}+\frac{1}{r})\theta\right)i+\varepsilon},
& \text{ if } \frac{\ell}{r}< 4+i\left( \frac{1}{2}+\frac{1}{r}\right).
\end{cases}
\eea
Similarly, the second term in \eqref{R-estimate} can be estimated as
\bea\label{eq:2 2}
&&\Delta X^{\frac{3+i}{2}+\varepsilon}\sum_{q\le P}q^{-\frac{\ell-i}{r}+\frac{i}{2}+3}\int_{\frac{1}{X}<|\beta |\le \frac{1}{qQ}} \;|\beta|^{\frac{3+i}{2}-\frac{\ell-i}{r}}\mathrm{d}\beta\nonumber\\
&\ll&\Delta X^{\frac{3+i}{2}+\varepsilon}\sum_{q\le P}q^{-\frac{\ell-i}{r}+\frac{i}{2}+3}
\begin{cases}
\left(\frac{1}{X}\right)^{\frac{5+i}{2}-\frac{\ell-i}{r}},
& \text{ if } \frac{\ell}{r}\ge \frac{5}{2}+i\left( \frac{1}{2}+\frac{1}{r}\right), \\
\\
\left(\frac{1}{qQ}\right)^{\frac{5+i}{2}-\frac{\ell-i}{r}},
& \text{ if } \frac{\ell}{r}< \frac{5}{2}+i\left( \frac{1}{2}+\frac{1}{r}\right),
\end{cases}\nonumber\\
&\ll&
\begin{cases}
\Delta X^{\frac{\ell-i}{r}-1+\varepsilon},
& \text{ if } \frac{\ell}{r}\ge 4+i\left( \frac{1}{2}+\frac{1}{r}\right),\nonumber\\
\\
\Delta X^{\frac{\ell}{r}-1-\left(\frac{\ell}{r}-4\right)\theta-\left(\frac{1}{r}-(\frac{1}{2}+\frac{1}{r})\theta\right)i+\varepsilon},
& \text{ if } \frac{5}{2}+i\left( \frac{1}{2}+\frac{1}{r}\right)\le \frac{\ell}{r}< 4+i\left( \frac{1}{2}+\frac{1}{r}\right),\nonumber\\
\\
\Delta X^{\frac{\ell}{r}-1-\left(\frac{\ell}{r}-4\right)\theta-\left(\frac{1}{r}-(\frac{1}{2}+\frac{1}{r})\theta\right)i+\varepsilon},
& \text{ if } \frac{\ell}{r}< \frac{5}{2}+i\left( \frac{1}{2}+\frac{1}{r}\right),
\end{cases}\nonumber\\
&\ll&
\begin{cases}
\Delta X^{\frac{\ell-i}{r}-1+\varepsilon},
& \text{ if } \frac{\ell}{r}\ge 4+i\left( \frac{1}{2}+\frac{1}{r}\right),\\
\\
\Delta X^{\frac{\ell}{r}-1-\left(\frac{\ell}{r}-4\right)\theta-\left(\frac{1}{r}-(\frac{1}{2}+\frac{1}{r})\theta\right)i+\varepsilon},
& \text{ if } \frac{\ell}{r}< 4+i\left( \frac{1}{2}+\frac{1}{r}\right),
\end{cases}
\eea
and the last term in \eqref{R-estimate} can be estimated as
\bea\label{eq:33}
&&X^{\frac{5+i}{2}+\varepsilon}\sum_{q\le P}q^{-\frac{\ell-i}{r}+\frac{i}{2}+3}\int_{\frac{1}{X}<|\beta |\le \frac{1}{qQ}} \;|\beta|^{\frac{5+i}{2}-\frac{\ell-i}{r}}\mathrm{d}\beta\nonumber\\
&\ll&
X^{\frac{5+i}{2}+\varepsilon}\sum_{q\le P}q^{-\frac{\ell-i}{r}+\frac{i}{2}+3}
\begin{cases}
\left(\frac{1}{X}\right)^{\frac{7+i}{2}-\frac{\ell-i}{r}},
& \text{ if } \frac{\ell}{r}\ge \frac{7}{2}+i\left( \frac{1}{2}+\frac{1}{r}\right), \\
\\
\left(\frac{1}{qQ}\right)^{\frac{7+i}{2}-\frac{\ell-i}{r}},
& \text{ if } \frac{\ell}{r}< \frac{7}{2}+i\left( \frac{1}{2}+\frac{1}{r}\right),
\end{cases}\nonumber\\
&\ll&
\begin{cases}
X^{\frac{\ell-i}{r}-1+\varepsilon},
& \text{ if } \frac{\ell}{r}\ge 4+i\left( \frac{1}{2}+\frac{1}{r}\right), \\
\\
X^{\frac{\ell}{r}-1-\left(\frac{\ell}{r}-4\right)\theta-\left(\frac{1}{r}-(\frac{1}{2}+\frac{1}{r})\theta\right)i+\varepsilon},
& \text{ if } \frac{7}{2}+i\left( \frac{1}{2}+\frac{1}{r}\right)\le \frac{\ell}{r}< 4+i\left( \frac{1}{2}+\frac{1}{r}\right),\nonumber\\
\\
X^{\frac{\ell}{r}-1-\left(\frac{\ell}{r}-4\right)\theta-\left(\frac{1}{r}-(\frac{1}{2}+\frac{1}{r})\theta\right)i+\varepsilon},
& \text{ if } \frac{\ell}{r}< \frac{7}{2}+i\left( \frac{1}{2}+\frac{1}{r}\right),
\end{cases}\\
&\ll&
\begin{cases}
X^{\frac{\ell-i}{r}-1+\varepsilon},
& \text{ if } \frac{\ell}{r}\ge 4+i\left( \frac{1}{2}+\frac{1}{r}\right), \\
\\
X^{\frac{\ell}{r}-1-\left(\frac{\ell}{r}-4\right)\theta-\left(\frac{1}{r}-(\frac{1}{2}+\frac{1}{r})\theta\right)i+\varepsilon},
& \text{ if } \frac{\ell}{r}< 4+i\left( \frac{1}{2}+\frac{1}{r}\right).
\end{cases}
\eea
By \eqref{R-estimate}-\eqref{eq:33},
$\mathbf{R}_i^{\pm}$ can be bounded as
\bea\label{R-estimate2}
\mathbf{R}_i^{\pm}\ll
\begin{cases}
\Delta X^{\frac{\ell}{r}-1-\left(\frac{\ell}{r}-4\right)\theta-\left(\frac{1}{r}-(\frac{1}{2}+\frac{1}{r})\theta\right)i+\varepsilon},
& \text{ if } \frac{\ell}{r}< 4+i\left( \frac{1}{2}+\frac{1}{r}\right), \\
\\
\Delta X^{\frac{\ell-i}{r}-1+\varepsilon},
& \text{ if } \frac{\ell}{r}\ge 4+i\left(\frac{1}{2}+\frac{1}{r}\right),
\end{cases}
\eea

Inserting \eqref{M-estimate2} and \eqref{R-estimate2} into \eqref{med-estimate}, we obtain
\bea\label{eq:Major}
&&\int_\mathfrak{M}\mathscr{F}_r^\ell(\alpha, X )\mathscr{F}_s(\alpha, X)\mathscr{G}(\alpha, X )\mathrm{d}\alpha\nonumber\\
&\ll&\Delta X^{\frac{\ell}{r}+\frac{1}{s}-1-\left(\frac{\ell}{r}+\frac{1}{s}-\frac{7}{2}\right)\theta+\varepsilon}+\Delta X^{\frac{\ell}{r}-1-\left(\frac{\ell}{r}-4\right)\theta+\varepsilon}+\Delta  X^{\frac{\ell}{r}+\frac{1}{s}-1+\varepsilon}.
\eea
This completes the treatment of the major arcs.

\section{Completion of the proof of Theorem \ref{Th11}}
\setcounter{equation}{0}

Recall that $r\geq2$, $s\geq 2$, $\ell\geq 2^{r-1}$ and by \eqref{theta}
$$\theta\le\theta_0=\min\big\{1/r, 1/s\big\}.$$

\medskip

\quad  (\romannumeral1) $2\le r\le 7$, $2\le s\le 7$.

By \eqref{minor arcs} and \eqref{eq:Major}, \eqref{deposition} is bounded by
\bna
\mathscr{S}(X)
&\ll& \Delta X^{\frac{\ell}{r}+\frac{1}{s}-1-\left(\frac{\ell}{r}+\frac{1}{s}-\frac{7}{2}\right)\theta+\varepsilon}+\Delta X^{\frac{\ell}{r}-1-\left(\frac{\ell}{r}-4\right)\theta+\varepsilon}\\
&&+\Delta  X^{\frac{\ell}{r}+\frac{1}{s}-1+\varepsilon}+X^{\frac{\ell}{r}+\frac{1}{s}-\left(\frac{\ell}{2^{r-1}}+\frac{1}{2^{s-1}}-1\right)\theta+\varepsilon}.
\ena

For $\frac{\ell}{r}+\frac{1}{s}\geq \frac{7}{2}$, we take $\theta=\theta_0$ and then
\bea\label{1}
\mathscr{S}(X)\ll \Delta  X^{\frac{\ell}{r}+\frac{1}{s}-1+\varepsilon}+X^{\frac{\ell}{r}+\frac{1}{s} -\left(\frac{\ell}{2^{r-1}}+\frac{1}{2^{s-1}}-1\right)\theta_0+\varepsilon}.
\eea

For $\frac{\ell}{r}+\frac{1}{s}< \frac{7}{2}$, we first note that, for $s\geq 2$ and $\theta\leq 1/s$,
\bea\label{bj}
\frac{\ell}{r}+\frac{1}{s}-1-\left(\frac{\ell}{r}+\frac{1}{s}-\frac{7}{2}\right)\theta
\geq \frac{\ell}{r}-1-\left(\frac{\ell}{r}-4\right)\theta.
\eea
Thus by taking
$\theta=\theta_0$, we obtain
\bea\label{2}
\mathscr{S}(X)\ll \Delta X^{\frac{\ell}{r}+\frac{1}{s}-\left(1-\left(\frac{7}{2}-\frac{\ell}{r}-\frac{1}{s}\right)\theta_0\right)+\varepsilon}
+X^{\frac{\ell}{r}+\frac{1}{s}-\left(\frac{\ell}{2^{r-1}}+\frac{1}{2^{s-1}}-1\right)\theta_0+\varepsilon}.
\eea

\medskip

\quad  (\romannumeral2) $2\le r\le 7$, $s\ge 8$.

By \eqref{minor arcs2} and \eqref{eq:Major}, \eqref{deposition} is bounded by
\bna
\mathscr{S}(X)
&\ll& \Delta X^{\frac{\ell}{r}+\frac{1}{s}-1-\left(\frac{\ell}{r}+\frac{1}{s}-\frac{7}{2}\right)\theta+\varepsilon}+\Delta X^{\frac{\ell}{r}-1-\left(\frac{\ell}{r}-4\right)\theta+\varepsilon}\\
&&+\Delta X^{\frac{\ell}{r}+\frac{1}{s}-1+\varepsilon}+X^{\frac{\ell}{r}+\frac{1}{s}-\left(\frac{\ell}{2^{r-1}}+\frac{1}{2s(s-1)}-1\right)\theta+\varepsilon}.
\ena

For $\frac{\ell}{r}+\frac{1}{s}\geq \frac{7}{2}$, we take $\theta=\theta_0$ and then
\bea\label{3}
\mathscr{S}(X)\ll \Delta  X^{\frac{\ell}{r}+\frac{1}{s}-1+\varepsilon}+X^{\frac{\ell}{r}+\frac{1}{s}
-\left(\frac{\ell}{2^{r-1}}+\frac{1}{2s(s-1)}-1\right)\theta_0+\varepsilon}.
\eea

For $\frac{\ell}{r}+\frac{1}{s}< \frac{7}{2}$, notice the inequality in \eqref{bj} and by taking
$\theta=\theta_0$, we have
\bea\label{4}
\mathscr{S}(X)\ll \Delta X^{\frac{\ell}{r}+\frac{1}{s}-\left(1-\left(\frac{7}{2}-\frac{\ell}{r}-\frac{1}{s}\right)\theta_0\right) +\varepsilon}+X^{\frac{\ell}{r}+\frac{1}{s}-\left(\frac{\ell}{2^{r-1}}+\frac{1}{2s(s-1)}-1\right)\theta_0+\varepsilon}.
\eea

\medskip

\quad  (\romannumeral3) $r\ge8$, $2\le s\le 7$.

By \eqref{romannumeral3} and \eqref{eq:Major}, \eqref{deposition} is bounded by
\bna
\mathscr{S}(X)
&\ll& \Delta X^{\frac{\ell}{r}+\frac{1}{s}-1-\left(\frac{\ell}{r}+\frac{1}{s}-\frac{7}{2}\right)\theta+\varepsilon}+\Delta X^{\frac{\ell}{r}-1-\left(\frac{\ell}{r}-4\right)\theta+\varepsilon}\\
&&+\Delta  X^{\frac{\ell}{r}+\frac{1}{s}-1+\varepsilon}+X^{\frac{\ell}{r}+\frac{1}{s}-\left(\frac{\ell-2^{r-1}}{2r(r-1)}+\frac{1}{2^{s-1}}\right)\theta+\varepsilon}.
\ena

For $\frac{\ell}{r}+\frac{1}{s}\geq \frac{7}{2}$, we take $\theta=\theta_0$ and then
\bea\label{5}
\mathscr{S}(X)\ll \Delta  X^{\frac{\ell}{r}+\frac{1}{s}-1+\varepsilon}+X^{\frac{\ell}{r}+\frac{1}{s}-\left(\frac{\ell-2^{r-1}}{2r(r-1)}+\frac{1}{2^{s-1}}\right)\theta_0+\varepsilon}.
\eea

For $\frac{\ell}{r}+\frac{1}{s}< \frac{7}{2}$, using \eqref{bj} and taking
$\theta=\theta_0$, we have
\bea\label{6}
\mathscr{S}(X)\ll\Delta X^{\frac{\ell}{r}+\frac{1}{s}-\left(1-\left(\frac{7}{2}-\frac{\ell}{r}-\frac{1}{s}\right)\theta_0\right)+\varepsilon}
+X^{\frac{\ell}{r}+\frac{1}{s}-\left(\frac{\ell-2^{r-1}}{2r(r-1)}+\frac{1}{2^{s-1}}\right)\theta_0+\varepsilon}.
\eea

\medskip

\quad  (\romannumeral4) $r\ge8$, $s\ge 8$.

By \eqref{romannumeral4} and \eqref{eq:Major}, \eqref{deposition} is bounded by
\bna
\mathscr{S}(X)&\ll& \Delta X^{\frac{\ell}{r}+\frac{1}{s}-1-\left(\frac{\ell}{r}+\frac{1}{s}-\frac{7}{2}\right)\theta+\varepsilon}+\Delta X^{\frac{\ell}{r}-1-\left(\frac{\ell}{r}-4\right)\theta+\varepsilon}\\
&&+\Delta  X^{\frac{\ell}{r}+\frac{1}{s}-1+\varepsilon}+X^{\frac{\ell}{r}+\frac{1}{s}-\left(\frac{\ell-2^{r-1}}{2r(r-1)}+\frac{1}{2s(s-1)}\right)\theta+\varepsilon}.
\ena

For $\frac{\ell}{r}+\frac{1}{s}\geq \frac{7}{2}$, we take $\theta=\theta_0$ and then
\bea\label{7}
\mathscr{S}(X)\ll \Delta  X^{\frac{\ell}{r}+\frac{1}{s}-1+\varepsilon}+X^{\frac{\ell}{r}+\frac{1}{s}
-\left(\frac{\ell-2^{r-1}}{2r(r-1)}+\frac{1}{2s(s-1)}\right)\theta_0+\varepsilon}.
\eea
For $\frac{\ell}{r}+\frac{1}{s}< \frac{7}{2}$, using \eqref{bj} and taking $\theta=\theta_0$, we have
\bea\label{8}
\mathscr{S}(X)
\ll \Delta X^{\frac{\ell}{r}+\frac{1}{s}-\left(1-
\left(\frac{7}{2}-\frac{\ell}{r}-\frac{1}{s}\right)\theta_0\right)+\varepsilon}+X^{\frac{\ell}{r}+\frac{1}{s} -\left(\frac{\ell-2^{r-1}}{2r(r-1)}+\frac{1}{2s(s-1)}\right)\theta_0+\varepsilon}.
\eea
By \eqref{1} and \eqref{2}--\eqref{8}, the proof of Theorem \ref{Th11} is complete.

\bigskip
\noindent{\bf Acknowledgements}
The authors are very grateful to the referees for their valuable suggestions.

\bigskip

{\small \textsc{Qingfeng Sun},
	\textsc{School of Mathematics and Statistics, Shandong University, Weihai,
		Weihai, Shandong 264209, China}\\
	\indent{\it E-mail address}: qfsun@sdu.edu.cn}

\medskip

{\small \textsc{Yanxue Yu},
\textsc{School of Mathematics and Statistics, Shandong University, Weihai,
		Weihai, Shandong 264209, China}\\
    \indent{\it E-mail address}: yanxueyu@mail.sdu.edu.cn}
	
\end{document}